\documentclass[12pt,   a4paper]{article}
\usepackage{amsfonts}
\usepackage{mathrsfs}
\usepackage{latexsym}
\usepackage{xy}
\usepackage{tikz-cd}
\usepackage{graphics}
\usepackage{amsfonts,  amsmath,  amssymb,  amsthm}
\usepackage{color}

\DeclareMathOperator{\FC}{FC}
\DeclareMathOperator{\GFC}{GFC}
\DeclareMathOperator{\Flat}{Flat}
\DeclareMathOperator{\Ker}{Ker}
\DeclareMathOperator{\Ext}{Ext}
\DeclareMathOperator{\Hom}{Hom}

\xyoption{all}

\newcommand{\bcen}{\begin{center}}     \newcommand{\ecen}{\end{center}}
\newcommand{\bay}{\begin{array}}      \newcommand{\eay}{\end{array}}
\newcommand{\beq}{\begin{eqnarray*}}      \newcommand{\eeq}{\end{eqnarray*}}

\def\Hom{\mathrm{Hom}}
\def\Ker{\mathrm{Ker}}
\def\Im{\mathrm{Im}}
\def\Ext{\mathrm{Ext}}
\def\Tor{\mathrm{Tor}}

\def\Mod{\mathrm{Mod}}
\def\FC{\mathrm{FC}}
\def\GFC{\mathrm{GFC}}

\def\Id{\mathrm{Id}}
\def\tr{\mathrm{tr}}
\def\pd{\mathrm{pd}}
\def\fd{\mathrm{fd}}

\def\Proj{\mathrm{Proj}}
\def\Flat{\mathrm{Flat}}

\def\Coker{\mathrm{Coker}}

\def\Ind{{\rm Ind}}

\def\arrow{{\rm arrow}}
\begin{document}

\newtheorem{theorem}{Theorem}[section]
\newtheorem{proposition}[theorem]{Proposition}
\newtheorem{lemma}[theorem]{Lemma}
\newtheorem{corollary}[theorem]{Corollary}
\newtheorem{remark}[theorem]{Remark}
\newtheorem{example}[theorem]{Example}
\newtheorem{definition}[theorem]{Definition}
\newtheorem{question}[theorem]{Question}
\numberwithin{equation}{section}

\title{\large\bf
Gorenstein flat-cotorsion modules over tensor rings}

\author{\large Yongyun Qin, Chaobin Yin}

\date{\footnotesize School of Mathematics,  \  Yunnan Key Laboratory of Modern Analytical Mathematics and Applications,  
Yunnan Normal University,   Kunming,   Yunnan 650500,   China.
\\E-mail: qinyongyun2006@126.com
	}

\maketitle

\begin{abstract}
Let  $T_R(M)$  be a tensor ring, where $R$ is a ring and $M$ is an
$N$-nilpotent  $R$-bimodule. 
Under certain  conditions,  we characterize the Gorenstein flat-cotorsion modules over $T_R(M)$,  showing that a $T_R(M)$-module $(X,  u)$ is Gorenstein flat-cotorsion if and only if $u$ is monomorphic and $\Coker u$  is a Gorenstein flat-cotorsion $R$-module. As applications, we describe
the Gorenstein flat-cotorsion modules over some trivial
extension rings and Morita context rings.
\end{abstract}

\medskip


\medskip


\bigskip

\section{\large Introduction}

\indent\indent
The study of Gorenstein homological algebra stems from finitely generated modules of G-dimension zero over any noetherian ring,    introduced by Auslander and
Bridger as a generalization of finite generated projective modules \cite{Ase22}.  
Later, Enochs and Jenda introduced Gorenstein projective modules (not necessarily finitely generated) over any associative ring \cite{EG},   and dually,  the Gorenstein injective modules were defined. Another pursuit 
of this analogy led to the concept of the Gorenstein flat modules introduced by Enochs et al. in \cite{EG4}. Modeled on the classical projective, injective and flat modules, the three types of Gorenstein homological modules have become a cornerstone of Gorenstein homological algebra \cite{EJ2000,  H2004,  HJ2004}.

 Recently, Gorenstein flat modules that are also cotorsion have attracted increasing investigation. This focus is largely due to the fact that the category of Gorenstein flat modules is not Frobenius, whereas the subcategory of Gorenstein flat and cotorsion modules is.
Motivated by this, Christensen et al. introduced the notion of Gorenstein flat-cotorsion modules (GFC, for short), and they showed that over a right coherent ring, GFC are precisely 
the modules that are Gorenstein flat and cotorsion \cite{{CEP}}. We refer the readers to \cite{B25, CELT, L26, LM26, LMY25}
for more discussions on Gorenstein flat-cotorsion modules.

Throughout the paper, all rings are nonzero associative rings with identity and all modules are unitary. For a ring $R$, we adopt the convention that an $R$-module is a left $R$-module; right $R$-modules are viewed as modules over the opposite ring $R^{op}$. Let $R$ be such a ring and $M$ be an $R$-bimodule.  Following \cite{CPM1991}, a tensor ring of an $R$-bimodule $M$ is defined as $ T_R(M) = \bigoplus_{i=0}^{\infty} M^i$, where $M^0=R$ and $M^{i+1}=M\otimes _RM^i$ for $i\geqslant 0$. The addition of elements in $T_R(M)$ is componentwise,
 and the multiplication is induced by the isomorphism
 $(\bigoplus_{n=0}^{\infty} M^n )\otimes_R (\bigoplus_{m=0}^{\infty} M^m) \cong \bigoplus_{n,m=0}^{\infty} M^{n+m }$.  Examples of tensor rings include but are not limited to trivial
 extension rings, Morita contex trings, triangular matrix rings and so on. In this paper, we focus on the case that $M$ is nilpotent, that is, $M^{N+1}=0$ for some $N\geq 0$. Then a module over $T_R(M)$ can be written as a pair $(X, u)$ with $X$ an $R$-module and $u\in \Hom_{R}(M\otimes _RX,X)$.
  The classical homological properties of the tensor ring  $T_R(M)$ were studied in \cite{MH,RYV}. Recently, various Gorenstein homological modules over $T_R(M)$ have been characterized, including Gorenstein projective (injective, flat) modules \cite{C-L, DLST}, Ding projective (injective) modules \cite{ZS} and projectively coresolved Gorenstein flat modules \cite{TJ}. On the other hand, Wang-He-Jin investigated Gorenstein flat-cotorsion modules over trivial extension rings \cite{WHJ}. This development naturally motivates the central question of this work.

\textbf{Question.} How to describe Gorenstein flat-cotorsion module over $T_R(M)$ for an $N$-nilpotent $R$-bimodule $M$?

The main purpose of the paper is to answer this question and to provide further applications. We need the following Tor-vanishing
condition on $M$ introduced by Chen-Lu \cite{C-L}:
$${\rm (P)}\quad\quad\quad \ \ \ \ \Tor_{\ge 1}^R(M,    M^i) = 0 \ {\rm for \ all}\ i\geq1. $$ 
Under this condition, our results can be stated as follows.

\begin{theorem}\label{thm-1.1}{\rm( Theorem ~\ref{cor-ten-GFC})}
 Let $M$ be an $N$-nilpotent $R$-bimodule which satisfies the condition  {\rm(P)}.  If $M \otimes_R-$ perserves cotorsion modules,
  $\fd_R M < \infty$ and $\fd _{R^{op}}M < \infty$, then a $T_R(M)$-module $(X,  u)$ is Gorenstein flat-cotorsion if and only if $u$ is monomorphic and $\Coker u$  is a Gorenstein flat-cotorsion $R$-module. 
 \end{theorem}
 
 Example~\ref{ex} shows that there exists a bimodule $M$ over some algebra satisfying
 all the conditions in the above theorem. The idea of our proof of Theorem~\ref{thm-1.1} is similar to that
in \cite{DLST}, but additional ingredients are needed.
Indeed, the proof of \cite[Proposition 2.11]{DLST} relies
 on the fact that every module is an image of a projective module and that the subcategory of
 Gorenstein projective modules is closed under kernels of epimorphisms. However,
 these properties do not hold for Gorenstein flat-cotorsion modules, and thus new arguments must be introduced to establish Theorem~\ref{thm-1.1}.
 As an application, we characterize the Gorenstein flat-cotorsion modules over some $1$-nilpotent trivial
extension rings and Morita context rings.
 
The structure of the paper is as follows. Section 2 collects basic material on tensor rings and introduces notation used throughout the paper. Section 3 is devoted to the proof of Theorem~\ref{thm-1.1}. Finally, Section 4 presents the applications.

\section{\large Prelimilaries}
\indent\indent In this section we fix our notation, recall relevant concepts, and present several basic facts. Throughout the paper, we denote by $\Mod (R)$ the category of left $R$-modules, 
and by $\Proj (R)$ (resp. $\Flat (R)$) the subcategory of $\Mod(R)$ consisting
of all projective (resp. flat) $R$-modules. For an $R$-module $X$, we denote by
$\pd(X)$ (resp. $\fd(X)$) the projective (resp. flat) dimension of $X$.

\noindent\textbf {2.1 Tensor rings.} Let $R$ be a ring and $M$ an $R$-bimodule,  a tensor ring of an R-bimodule $M$ is defined by $T_R(M) = \bigoplus_{i=0}^{\infty} M^i$, where $M^0=R$  and $M^{i+1}=M\otimes_R M^i $
 for $i\geqslant 0$.  In this paper,  we assume that $M$ is {\it $N$-nilpotent},
 that is, $ M^{N+1}=0$ for some $N\geqslant 0$, and
 we denote by $T_R(M)=\bigoplus_{i=0}^N M^i $, the {\it tensor ring} with respect to $M$.

From \cite{C-L}, a module over $ T_{R}(M)$ can be written as a pair $ (X,  u) $, where $ X \in\Mod(R) $ and $ u:M\otimes_{R}X\rightarrow X $ is an $R$-homomorphism. A morphism from $(X,   u)$ to $(X',   u')$ is an $R$-homomorphism $f \in \mathrm{Hom}_R(X,   X')$ such that  $f \circ u = u' \circ (M \otimes f)$.
A sequence
\[
(X,  u)\stackrel{f}{\longrightarrow}(X',  u')\stackrel{g}{\longrightarrow}(X'',  u'')
\]
in $\mathrm{Mod}(T_{R}(M))$ is exact if and only if the sequence $X\stackrel{f}{\longrightarrow}X'\stackrel{g}{\longrightarrow}X''$ is exact in $\Mod(R)$.

\noindent\textbf {2.2 Adjoint functors.} From \cite{C-L}, there are two adjoint pairs as follows:
\[
\begin{tikzcd}
\Mod(R) \arrow[r,   shift left=-0.7ex,   "S"'] & 
		\Mod(T_{R}(M))  \arrow[l,   shift left=-0.7ex,   "C"'] \arrow[r,   shift left=-0.7ex,   "U"'] &
	\Mod(R). \arrow[l,   shift left=-0.7ex,   "\Ind"']
\end{tikzcd}
\]
Here, $U$ is the forgetful functor
which maps a $T_{R}(M)$-module $(X,  u)$ to the underlying $R$-module $X$,
and for an $R$-module $X$,
$\Ind(X)$ is defined as $(\bigoplus_{i=0}^{N}(M^i\otimes_R X),  c_{X})$, where $c_{X}$ is the inclusion from $M\otimes_{R}(\bigoplus_{i=0}^{N}M^i\otimes_R X)\cong\bigoplus_{i=1}^{N}(M^i\otimes_{R}X)$ 
	to $\bigoplus_{i=0}^{N}(M^i\otimes_{R}X)$. More explicitly,  
	\[
	c_{X}=\begin{pmatrix}
		0 & 0 & 0 & \cdots & 0\\
		1 & 0 & 0 & \cdots & 0\\
		0 & 1 & 0 & \cdots & 0\\
		\vdots & \vdots & \vdots & \ddots & \vdots\\
		0 & 0 & 0 & \cdots & 1
	\end{pmatrix}_{(N+1)\times N}.
	\]
 For an $R$-homomorphism $f:X\to Y$,   the $T_{R}(M)$-homomorphism 
	$\Ind(f):\Ind(X)\to\Ind(Y)$ can be viewed as a formal diagonal matrix with 
	diagonal elements $M^i\otimes f$ for $0\leqslant i\leqslant N$.
	Moreover, the functor $S$
sends an $R$-module $X$ to the $T_{R}(M)$-module $(X,  0)$, 
and for a $T_{R}(M)$-module $(X,  u)$,  $C((X,  u))$ is defined as $\Coker u$.
 For a morphism $f:(X,  u)\to(Z,  w)$ in $\Mod(T_{R}(M))$,   the morphism 
	$C(f):\Coker u\to\Coker w$ is determined by the universal property of cokernels.

It is straightforward to verify that $C\circ\Ind=\Id_{\Mod(R)}$,   and by the Eilenberg-Watts theorem,   the functor 
$\Ind $ is naturally isomorphic to
$T_{R}(M)\otimes_{R}-$,  so one has $\Ind(X)\cong T_{R}(M)\otimes_{R}X$ 
for each $X\in \Mod(R)$. By \cite[Remark 1.5]{DLST},
a $T_{R}(M)$-module $(X,  u)$ gives rise to the following short exact sequence:
\[
0\to\Ind(M\otimes_{R}X)\xrightarrow{\phi_{(X,  u)}}\Ind(X)\xrightarrow{\varepsilon_{(X,  u)}}(X,  u)\to 0
\]
of $T_{R}(M)$-modules,  where
\[
\phi_{(X,  u)}=
\begin{pmatrix}
	-u & 0 & 0 & \cdots & 0\\
	1 & -M\otimes u & 0 & \cdots & 0\\
	0 & 1 & -M^2\otimes u & \cdots & 0\\
	\vdots & \vdots & \vdots & \ddots & \vdots\\
	0 & 0 & 0 & \cdots & 1
\end{pmatrix}_{(N+1)\times N},
\]
and
\[
\varepsilon_{(X,   u)} = 
\begin{pmatrix}
	1 &
	u &
	u \circ (M \otimes u) &
	\cdots &
	u \circ (M \otimes u) \circ \cdots \circ \left( M^{ N-1} \otimes u \right)
\end{pmatrix}.
\]

\noindent\textbf {2.3 Gorenstein flat-cotorsion modules.}
From \cite{Xu96}, an $R$-module $C$ is said to be {\it cotorsion} if $\Ext^1_R(F, C) = 0$ for any $F \in \Flat(R)$. For simplicity, we denote by
$${\rm FC(R)} = \{ M \in \Mod(R) \mid M \text{ is flat and cotorsion} \}.$$ 
Recall from \cite{CEP} that an exact complex $T^\bullet$
with $T^i \in \FC(R)$
is a {\it totally acyclic complex of flat-cotorsion
modules} if it remains exact after applying the functor
$\Hom_{R}(-,W)$ for every $W\in \FC(R)$. An $R$-module
$X$ is called
{\it Gorenstein flat-cotorsion} ($\GFC$, for short) if there is a totally acyclic complex
$T^\bullet$ of flat-cotorsion modules with
$X\cong\Ker(T^{0}\to T^{1})$. By \cite{CELT,CEP}, a Gorenstein flat and cotorsion module is 
GFC, and the converse also holds if $R$ is right coherent. However, this is not true for arbitrary rings.

We now record several auxiliary results that will be used frequently.
\begin{lemma}\label{P-F} For a tensor ring $T_{R}(M)$, the following equalities hold.

{\rm (1)(\cite[Lemma 1.9]{DLST})} $\Proj(T_{R}(M)) = \Ind(\Proj(R))$.
		
{\rm (2)(\cite[Lemma 2.8]{{ZS}})} $\Flat(T_{R}(M)) = \Ind(\Flat(R))$.
\end{lemma}

\begin{lemma}\label{Ext}{\rm (\cite [Lemma 3.7]{CRZ})}
	Let $R$ and $S$ be rings,   $(F,   G)$ an adjoint pair with $F: \Mod(R) \to \Mod(S)$.
	For an $X \in\Mod(R)$,   if there is a short exact sequence $0 \to K \to P \to X \to 0$ with $P\in \Proj(R)$ such that $0 \to FK \to FP \to FX \to 0$ is exact with $FP\in \Proj(S)$,  then
		$\operatorname{Ext}_{S}^{1}(FX,   Y) \cong \operatorname{Ext}_{R}^{1}(X,   GY)$   
for any $Y\in \Mod(S)$. 
\end{lemma}

\begin{lemma}\label{TG}
 For any $F\in\Flat(R)$ and $(X,  u)\in\Mod(T_{R}(M))$,   we have
\[
\Ext_{T_{R}(M)}^{1}(\Ind(F),  (X,  u))\cong\Ext_{R}^{1}(F,  X).
\]
\end{lemma}
\begin{proof}
	Since $( \Ind,   U )$ is an adjoint pair with $U$
	exact, we get that the functor $\Ind$ preserves projective modules.
	For any $F\in\Flat(R)$, consider the following short exact sequence
	\[
	0 \to K \to P \to F \to 0  
	\]
	with $P\in \Proj(R)$. Applying the right exact
	functor  $\Ind = T_R(M) \otimes_R -$, and noting that $\Tor_1^R(T_R(M),   F) = 0$, we have a short exact sequence
	\[
	0 \to \Ind(K) \to \Ind(P) \to \Ind(F) \to 0.
	\]
	Then the assertion follows from Lemma  \ref{Ext}.
\end{proof}

As a consequence, we get the following characterizations of flat and cotorsion
modules over $T_R(M)$.

\begin{lemma}\label{tensor-cor}
	$(X,   u)$ is a cotorsion  $T_R(M)$-module if and only if $X$ is a cotorsion $R$-module.
\end{lemma}
\begin{proof}
	It follows from Lemma \ref{P-F} and \ref{TG}.
\end{proof}

\begin{lemma}\label{tensor-FC}
If $(X,   u)\in \FC(T_R(M))$, then $(X,   u) \cong \operatorname{Ind}(F)$ for some $F\in \FC(R)$.
\end{lemma}

\begin{proof}
	Since $(X,   u)$ is flat, it follows from Lemma~\ref{P-F} 
that there exists $F \in \operatorname{Flat}(R)$ such that
$	(X,   u) \cong \operatorname{Ind}(F)$. It remains to show that
$F$ is cotorsion.
Since $(X,  u) \cong ( \bigoplus_{i=0}^{N} (M^{i} \otimes_{R} F),   c_{F} ) $
and $(X,   u)$ is cotorsion, it follows from Lemma~\ref{tensor-cor} that $ \bigoplus_{i=0}^{N} (M^{i} \otimes_{R} F)$ is also cotorsion, that is,  $
 \operatorname{Ext}_R^1 ( F',  \,   \bigoplus_{i=0}^N (M^i \otimes_R F) ) = 0
 $ for any $ F' \in \Flat(R)$. Therefore, we have
 $\operatorname{Ext}_R^1(F',   F) = 0$ for any $ F' \in \Flat(R)$, and thus
 $F$ is a cotorsion $R$-module. 
\end{proof}

\begin{lemma}\label{temsor-FC1}
	If $M \otimes_R-$ preserves cotorsion modules, then $\FC(T_R(M)) =\Ind(\FC(R))$.
\end{lemma}
\begin{proof}
By Lemma~\ref{tensor-FC}, we have the inclusion $ \FC(T_R(M)) \subseteq\Ind(\FC(R)) $. It remains to establish the converse inclusion, that is, $ \Ind(\FC(R)) \subseteq \FC(T_R(M)) $. To this end, let $X\in \FC(R)$. Then, by Lemma~\ref{P-F}, the induced module 
$\Ind(X)$ is flat. Now we claim that $\Ind(X)$ is also cotorsion.
Since $\Ind(X) = ( \bigoplus_{i=0}^N (M^i \otimes_R X),   \ c_X )$, it suffices, in view of Lemma~\ref{tensor-cor}, to prove that $\bigoplus_{i=0}^N (M^i \otimes_R X)$ is cotorsion, and this is due to the hypothesis that 
$M \otimes_R-$ preserves cotorsion modules.
\end{proof}

Next, we will give a sufficient condition under which the functor
$M \otimes_R-$ preserves cotorsion modules.

\begin{lemma}\label{lem-pre-cor}
If $M$ is a finitely-generated projective right $R$-module, then $M \otimes_R-$ preserves cotorsion modules.
\end{lemma}
\begin{proof}
Since $M$ is a finitely-generated projective right $R$-module, we have an isomorphism of functors
$M \otimes_R-\cong \Hom _R(\Hom _{R^{op}}(M,R),-)$, and thus $M \otimes_R-$ has a left adjoint
$\Hom _{R^{op}}(M,R)\otimes_R- $. Let $F=\Hom _{R^{op}}(M,R)\otimes_R-$ and $G=M \otimes_R-$.
Since $G$ is exact, we infer that $F$ preserves projective modules. Moreover, $F$
preserves the exactness of the exact sequence $0\rightarrow K\rightarrow P\rightarrow W\rightarrow 0$
with $W\in \Flat (R)$ since $\Tor _1^R(\Hom _{R^{op}}(M,R), W)=0$. Then it follows from
Lemma~\ref{Ext} that $$\operatorname{Ext}_{R}^{1}(FW,   X) \cong \operatorname{Ext}_{R}^{1}(W,   GX),$$
for any $W\in \Flat (R)$ and $X\in \Mod (R)$. Since $F$ preserves projective modules and filtered colimits,
we have that $F$ preserves flat modules. Now the above isomorphism implies that $G$ preserves cotorsion modules.      
\end{proof}

\section{Gorenstein flat-cotorsion modules over tensor rings}
\indent\indent In this section, we will give the proof Theorem~\ref{thm-1.1}.
In order to describe the Gorenstein flat-cotorsion modules over tensor rings, we need the following compatible and admissible conditions, which were first
defined in \cite{C-L, Z13} for
finite generated Gorenstein projective modules, and were further considered for infinite generated 
Gorenstein projective (injective) modules \cite{DLST} and $\mathcal{X}$-Gorenstein projective modules \cite{ZS}. 
Now we need the version of Gorenstein flat-cotorsion modules. 

\begin{definition} {\rm An $R$-bimodule $M$ is said to be \emph{compatible} if the following two conditions hold for a totally acyclic complex $T^\bullet$ of flat-cotorsion  $T_R(M)$-modules:

(C1) The complex $M \otimes_R T^\bullet$ is acyclic;
 
(C2) The complex $\Hom_{T_R(M)}(T^\bullet,   \Ind(M \otimes_R F))$ is acyclic for each $F \in \FC(R)$.}
\end{definition} 

\begin{definition} {\rm  An $R$-bimodule $M$ is said to be \textit{admissible} if
\[
\Ext_R^1(W,   M^i \otimes_R F) = 0 = \Tor_1^R(M,   M^i \otimes_R W)
\]
for each $W \in \GFC(R)$,   $F \in \FC(R)$ and $i \ge 0$.
}
\end{definition} 

From \cite{C-L}, the condition (P) for an $R$-bimodule $M$ is defined as follow:
$${\rm (P)}\quad\quad\quad \ \ \ \ \Tor_{\ge 1}^R(M,    M^i) = 0 \ {\rm for \ all}\ i\geq1. $$
Inspired by \cite[Proposition 2.7]{DLST}, we will give a sufficient condition for a module $M$ to be compatible and admissible. We need the following two lemmas.  
\begin{lemma}\label{Dim}
	{\rm (\cite[Theorem 9.48]{RG1979})}
	Let $R$ and $S$ be rings,   and $U$ an $R$-$S$-bimodule. If	$
	\fd(_R U) < \infty$, then	$
	\fd(_R(U\otimes_{S}F)) < \infty
	$ for any $F\in \Flat(S)$.
\end{lemma}

\begin{lemma}\label{DimI}
	Let $M$ be an $R$-bimodule satisfying the condition {\rm (P)}. If $\fd_R M < \infty$,   
	then $\fd_{T_R(M)}(\Ind(M\otimes_R F)) < \infty	$ for any $F\in \Flat(R)$.
\end{lemma}

\begin{proof}
Since $M$ satisfies
the condition (P), it follows from \cite[Lemma 4.2]{C-L} that $\Tor_{\geqslant 1}^R(M^s,   M \otimes_R F)=0$ for all $s\ge 1$
and $F\in \Flat(R)$, and thus
$\Tor_{\geqslant 1}^R(T_R (M),   M \otimes_R F)=0$.
On the other hand, it follows from Lemma $ \ref{Dim}$ that $\fd_R (M\otimes_R F\bigr) \le n $ for some $n$. Considering the flat resolution
	\[
	0 \to F^{-n} \to \cdots \to F^{-1} \to F^0 \to M \otimes_R F \to 0   
	\]
	and applying $\Ind = T_R(M) \otimes_R -$,  we obtain a complex of $T_R(M)$-modules
	\[
	0 \to \Ind(F^{-n}) \to \cdots \to \Ind(F^{-1}) \to \Ind(F^{0}) \to \Ind(M \otimes_R F) \to 0,
	\] which is also exact since $\Tor_{\geqslant 1}^R(T_R (M),   M \otimes_R F)=0$.
Now the fact $\Ind(F^{-i}) \in \Flat (T_R (M))$ implies that $\fd_{T_R(M)}(\Ind(M\otimes_R F))  < \infty$.
\end{proof}

Now we are ready to give a sufficient condition for a module $M$ to be compatible and admissible.
\begin{proposition}\label{com-adm}
Suppose that $M$ satisfies the condition ${\rm(P)}$ and $M \otimes_R-$
preserves cotorsion modules. If $\fd_R M < \infty$
and  $\operatorname{fd}_{R^{\mathrm{op}}} M < \infty$,  then $M$ is compatible and admissible.
\end{proposition}

\begin{proof}
	Let $T^\bullet$ be a totally acyclic complex of flat-cotorsion $T_R(M)$-modules.
	Since $T^i\in \FC (T_R(M))$, it follows from Lemma~\ref{tensor-FC} that
	$T^\bullet$ can be written as: 
	\[
	T^\bullet: \cdots \to \Ind(F^{-1}) \to \Ind(F^0) \to \Ind(F^1) \to \cdots ,
	\]
	where $F^i\in \FC(R)$ for each $i \in \mathbb{Z}$. Applying the functor $U$, we get the following exact sequence
$$Q^\bullet: \cdots \to \textstyle\bigoplus_{i=0}^N  (M^i \otimes_R F^{-1}) \to \bigoplus_{i=0}^N  ( M^i \otimes_R F^0) \to \bigoplus_{i=0}^N   (M^i \otimes_R F^1) \to \cdots.
$$

To prove the condition (C1), we have to prove that $M \otimes_R Q^\bullet$ is exact. Since $\operatorname{fd}_{R^{\mathrm{op}}} M < \infty$,   there is an exact sequence of
right $R$-modules
	\[
	0 \to \widetilde{ F^n} \to \widetilde{F^{n-1}} \to \cdots \to \widetilde{F^0} \to M \to 0  
	\] with $\widetilde{ F^i}\in \Flat(R^{\mathrm{op}})$.
	 Since $M$ satisfies
	the condition (P), it follows from \cite[Lemma 4.2]{C-L} that $\Tor^R_{\geqslant	1}(M,   M^i \otimes_R F) = 0$ for all $F\in \Flat(R)$ and $i\ge 0$, and thus
	$\Tor^R_{\geqslant	1}(M,   \bigoplus_{i=0}^N  (M^i \otimes_R F)) = 0$. Then there exists an exact sequence of complexes
	\[
	0 \to \widetilde{F^n} \otimes_R Q ^\bullet\to \widetilde{F^{n-1}} \otimes_R Q^\bullet \to \cdots \to \widetilde{F^0} \otimes_R Q^\bullet \to M \otimes_R Q^\bullet \to 0. 
	\]
Since $\widetilde{F^k} \otimes_R-$ is exact,  one gets that the complex $\widetilde{F^k} \otimes_R Q^\bullet$ is acyclic for $0 \le k \le n$, and thus $M \otimes_R Q^\bullet$ is exact, as desired.
	
We then prove the condition (C2). For any $F\in \FC(R)$, 
it follows from the fact $\fd_R M < \infty$ and Lemma \ref{DimI} that $\fd_{T_R(M)}(\Ind(M \otimes_R F)) < \infty$. Now we claim that $\Ind(M \otimes_R F)$ is cotorsion, and then
the complex $\Hom_{T_R(M)}(T^\bullet,   \Ind(M \otimes_R F))$ is exact by \cite[Lemma 2.1]{WHJ}.
Since $\Ind(M \otimes_R F) = ( \bigoplus_{i=1}^N (M^i \otimes_R F),   \ c_{M \otimes_R F} )$, it suffices, in view of Lemma~\ref{tensor-cor}, to prove that $\bigoplus_{i=1}^N (M^i \otimes_R F)$ is cotorsion, and this is due to the hypothesis that 
$M \otimes_R-$ preserves cotorsion modules.
	
Next, we prove that $M$ is admissible. Since $M$ satisfies
the condition (P), it follows from \cite[Lemma 4.2]{C-L} that $\Tor^R_{\geqslant	1}(M,   M^i \otimes_R P) = 0$ for all $P\in \Proj(R)$ and $i\ge 1$, and then \cite[Lemma 2.6] {DLST}
yields that
 $\fd _R(M^i) < \infty$ for any $i \ge 0$. Applying Lemma $\ref{Dim}$, we have $\fd _R(M^i \otimes_R F) < \infty$ for all $F\in \FC(R)$.
Since $M \otimes_R-$
preserves cotorsion modules, we get that $M^i \otimes_R F$ is also cotorsion.
Then it follows from \cite[Lemma 2.1]{WHJ} that
$\Ext_R^1(W,    M^i \otimes_R F) = 0$ for each $W \in \GFC(R)$,   $F \in \FC(R)$ and $i \ge 0$.

On the other hand, since $\fd_{R^{\mathrm{op}}} (M ^s) < \infty $ by \cite[Lemma 2.6] {DLST}, a standard dimension shifting argument yields 
that $\Tor^R_{\ge 1} (M ^s,   W) = 0$ for any  $W \in \GFC(R)$.
Thus, by \cite[Lemma 4.2]{C-L},   one gets that $\operatorname{Tor}^R_{\ge 1} (M,   M ^s\otimes_R W) = 0$ for any  $W \in \GFC(R)$ and $s\geqslant 0$. This yields that $M$ is admissible.
\end{proof}

In the following, we will characterize the Gorenstein flat-cotorsion modules over $T_R(M)$.
\begin{proposition}\label{prop:compatible}
{ Suppose that $M \otimes_R-$
	preserves cotorsion modules and $M$ is compatible. If $(X,   u)\in \GFC(T_R(M))$, then 
	$\Coker u\in \GFC(R)$ and $u$ is a monomorphism.
}
\end{proposition} 
\begin{proof}
	Let $(X,   u) \in \GFC(T_R(M))$. It follows from
Lemma~\ref{tensor-FC} that there is a totally acyclic complex 
of flat-cotorsion $T_R(M)$-modules 
	\[
	Q^\bullet: \ \ \cdots \xrightarrow{d^{-2}} \Ind(F^{-1}) \xrightarrow{d^{-1}} \Ind(F^0) \xrightarrow{d^0} \Ind(F^1)\xrightarrow{d^1} \cdots
	\]
	with $F^i\in \FC(R)$  and $(X,   u) \cong \ker(d^0)$.
Consider the monomorphism $\eta: (X,   u) \hookrightarrow \Ind(F^0)=(\bigoplus_{i=0}^N (M^i \otimes_R F^0),   c_{F^0})$. By the acyclicity of $M \otimes_R Q^\bullet$, we infer that
	$M \otimes \eta$ is a monomorphism. Since $c_{F^0}$ is a monomorphism,
	it follows from the commutative diagram
	\[
	\begin{tikzcd}
		M \otimes_R X \arrow[r,   "u"] \arrow[d,   "M \otimes \eta"'] & X \arrow[d,   "\eta"] \\
		\bigoplus_{i=1}^N M^i \otimes_R F^0 \arrow[r,   "c_{F^0}"] & \bigoplus_{i=0}^N M^i \otimes_R F^0
	\end{tikzcd}
	\]
	that $u$ is also a monomorphism.
	On the other hand, we have the following commutative
	diagram 	\[
	\begin{tikzcd}
		0 \arrow[r] & \bigoplus_{i=1}^{N} (M^i \otimes_R F^{-1}) \arrow[r,   "c_{F^{-1}}"] \arrow[d,   "M \otimes d^{-1}"'] & \bigoplus_{i=0}^{N} (M^i \otimes_R F^{-1}) \arrow[r,   "\pi_{F^{-1}}"] \arrow[d,   "d^{-1}"'] & F^{-1} \arrow[r] \arrow[d] & 0 \\
		0 \arrow[r] & \bigoplus_{i=1}^{N} (M^i \otimes_R F^0) \arrow[r,   "c_{F^0}"] \arrow[d,   "M \otimes d^0"'] & \bigoplus_{i=0}^{N} (M^i \otimes_R F^0) \arrow[r,   "\pi_{F^0}"] \arrow[d,   "d^0"'] & F^0 \arrow[r] \arrow[d] & 0 \\
		0 \arrow[r] & \bigoplus_{i=1}^{N} (M^i \otimes_R F^1) \arrow[r,   "c_{F^1}"] & \bigoplus_{i=0}^{N} (M^i \otimes_R F^1) \arrow[r,   "\pi_{F^1}"] & F^1 \arrow[r] & 0,
	\end{tikzcd}
	\] with exact rows. Since $U(Q^\bullet)$ and $M \otimes_R Q^\bullet$ are exact,
	we infer that $F^\bullet$ is also exact, and by Snake Lemma, we see that
	$\Coker u$ is isomorphic to the zeroth cocycle of $F^\bullet$. Therefore, to show
	$\Coker u\in \GFC(R)$, it is sufficient to show that $\Hom_R(F^\bullet,   \widetilde{F})$ is exact for each $\widetilde{F}\in \FC(R)$. By \cite[Remark 1.5]{DLST},
	there is a short exact sequence
\[
0 \to \Ind(M \otimes_R \widetilde{F}) \xrightarrow{\phi_{(\widetilde{F},  0)}} \Ind(\widetilde{F}) \xrightarrow{\varepsilon_{(\widetilde{F},  0)}} (\widetilde{F},  0) \to 0
\]
in $\Mod(T_R(M))$. 
Since $M \otimes_R-$
	preserves cotorsion modules, we deduce that $\Ind(M \otimes_R \widetilde{F})$ is cotorsion by
Lemma~\ref{tensor-cor}, and thus $\Ext _{T_R(M)}^1(Q^i, \Ind(M \otimes_R \widetilde{F}))=0$ as $Q^i\in \Flat(T_R(M))$.
Applying the functor $\Hom_{T_R(M)}(Q^\bullet,   -)$,   one gets a short exact sequence
\[
0 \longrightarrow (Q^\bullet,   \operatorname{Ind}(M \otimes_R \widetilde{F})) \xrightarrow{(Q^\bullet,   \phi{(\widetilde{F},  0)})} (Q^\bullet,   \operatorname{Ind}(\widetilde{F})) \xrightarrow{(Q^\bullet,   \eta{(\widetilde{F},  0)})} (Q^\bullet,   (\widetilde{F},  0)) \longrightarrow 0,
\]
where $(-,-):=\Hom_{T_R(M)}(-,   -)$. On the other hand, it follows from
Lemma~\ref{temsor-FC1} that $\Ind(\widetilde{F})\in\Ind(\FC(R))\subseteq \FC(T_R(M))$. Since $Q^\bullet$ is a totally acyclic complex of flat-cotorsion modules, we infer that
$\Hom_{T_R(M)}(Q^\bullet,   \Ind(\widetilde{F}))$ is exact. 
By assumption,   $\Hom_{T_R(M)}(Q^\bullet,   \Ind(M \otimes_R \widetilde{F}))$ is also exact. Then   $\Hom_{T_R(M)}(Q^\bullet,   S(\widetilde{F})) = \Hom_{T_R(M)}(Q^\bullet,   (\widetilde{F},   0))$ is exact,   which yields that $\Hom_R(C(Q^\bullet),   \widetilde{F})$ is exact. However,   the complex $C(Q^\bullet)$ is nothing but ${F}^\bullet$ since $F^i=\Coker (c_{F^i})$.  Consequently,   $\Hom_R(F^\bullet,   \widetilde{F})$ is acyclic,   as desired.
\end{proof}

Next, we will consider the converse of Proposition~\ref{prop:compatible}.
We need the following lemma.
\begin{lemma}\label{lem-ext-zero}
Suppose that $M \otimes_R-$
preserves cotorsion modules and $M$ is admissible. Let $(X,u)$ be a $T_R(M)$-module. If 
$u$ is a monomorphism  and $ \Coker(u) \in \GFC(R) $, then $\Ext_R^1(F,   M \otimes_R{X} ) = 0$
for any $F\in \Flat(R)$.
\end{lemma}
\begin{proof}
Since $M$ is admissible, we have  
$\Tor_1^R(M,   M^i \otimes_R \Coker  u)=0 $ for any $i \geq 0 $. Applying $M \otimes_R -$ to the short exact sequence
$$\delta: \ 0 \to M \otimes_R X \xrightarrow{u} X \xrightarrow{\rho} \Coker  u \to 0,$$  we get an exact sequence
$$M\otimes \delta:\  0 \to M^2 \otimes_R X  \xrightarrow{M \otimes u} M \otimes_R X \xrightarrow{M \otimes{\rho}} M\otimes_R \Coker  u \to 0$$
as $\Tor_1^R(M,   \Coker  u)=0 $.
Similarly, applying the functor $M \otimes_R -$ to the sequence $M\otimes \delta$, we get an exact sequence
$$M^2\otimes \delta :\  0 \to M^3 \otimes_R X \xrightarrow{M^2 \otimes u} M^2 \otimes_R X \xrightarrow{M^2 \otimes \rho} M^2 \otimes_R \Coker  u \to 0$$
as $\Tor_1^R(M,   M \otimes_R \Coker u) = 0$.
In the same way, for each $3\leq i\leq N$, we get exact sequences
$$M^i\otimes \delta : \ 0 \longrightarrow M^{i+1} \otimes_R X \xrightarrow{M^{i} \otimes u} M^{i} \otimes_R X \xrightarrow{M^{i} \otimes \rho} M^{i}\otimes_R\Coker u \longrightarrow 0.$$
Since $M$ is $N$-nilpotent, we have $M^{N} \otimes_R X\cong M^{N}\otimes_R\Coker u$. On the other hand, it follows from \cite[Theorem 1.3]{BIE20}
or \cite[Lemma 2.1]{WHJ} that every
Gorenstein flat-cotorsion module is cotorsion. Therefore, $\Coker u$ is cotorsion
and so is $M^{i}\otimes_R\Coker u$ since $M \otimes_R-$
preserves cotorsion modules. Then we have $\Ext_R^1(F,  M^{i}\otimes_R\Coker u) =0$
for any $F\in \Flat(R)$ and $i\geq 0$, and thus
$\Ext_R^1(F,  M^{N} \otimes_R X) \cong \Ext_R^1(F,  M^{N}\otimes_R\Coker u) =0$.
Now the exact    
sequence
$$M^{N-1}\otimes \delta :\  0 \rightarrow M^{N} \otimes_R X \xrightarrow{M^{N-1} \otimes u} M^{N-1} \otimes_R X \xrightarrow{M^{N-1} \otimes \rho} M^{N-1}\otimes_R\Coker u \rightarrow 0$$
yields that $\Ext_R^1(F,  M^{N-1} \otimes_R X)=0$.
Iteratively, by the exact sequences $M^{N-2}\otimes \delta , \cdots, M\otimes \delta $, we obtain that $\Ext_R^1(F,   M \otimes_R X) = 0 $.
\end{proof}

\begin{proposition}\label{prop:admissible}
Suppose that  $M \otimes_R-$
preserves cotorsion modules and $M$ is admissible. Let $(X,u)$ be a $T_R(M)$-module. If 
$u$ is a monomorphism  and $ \Coker(u) \in \GFC(R) $, then $(X,u)\in \GFC(T_R(M))$.
\end{proposition}
 
\begin{proof} Assume that $(X,u)\in \Mod(T_R(M))$ such that $ \Coker(u) \in \GFC(R) $ and $u$ is a monomorphism.
Then there is a totally acyclic complex of flat-cotorsion $R$-modules
$$F^\bullet: \ \ \cdots \rightarrow F^{-1}\xrightarrow{d^{-1}} F^{0} \xrightarrow{d^{0}}  F^{1}\rightarrow \cdots$$
such that $ \Coker(u) \cong \Ker (d^0)$.
Using a method similar to that in \cite[Proposition 2.11]{DLST}, the admissibility of 
$M$ yields an exact sequence of $T_R(M)$-modules $$0 \to (X,   u) \xrightarrow{\widetilde{a}} \Ind F^0 \xrightarrow{} \Ind F^1 \rightarrow \cdots, 
 $$ where $\widetilde{a}=(a^0,a^1,\cdots, a^N)^{\tr}:X\rightarrow \bigoplus_{i=0}^N (M^i \otimes_R F^{0})$,
 and the $R$-homomorphism $a^0: X\rightarrow  F^{0}$ 
 is the composition of the epimorphism $\rho: X\twoheadrightarrow \Coker(u)$
and the monomorphism $i:\Coker(u)\hookrightarrow F^0$.
Indeed, the argument for statement (1) in \cite[Proposition 2.11]{DLST} applies directly.
The second statement, however, does not hold, as not every module is an image of a flat cotorsion module,
and GFC is not closed under taking kernels of epimorphisms. 
So, it remains to show that there is an exact sequence $$\cdots  \rightarrow \Ind F^{-2} \rightarrow \Ind F^{-1} \xrightarrow{\widetilde{b}}  (X,   u) \rightarrow 0$$
in $\Mod(T_R(M))$.  

Applying $\Hom_R(F^{-1},-)$  to the exact sequence $ 0 \to M \otimes_R X \xrightarrow{u} X \xrightarrow{\rho} \Coker  u \to 0$,
and using Lemma~\ref{lem-ext-zero}, we get
that $\rho_*: \Hom_R(F^{-1},   X) \rightarrow \Hom_R(F^{-1},   \Coker  u) $ is an epimorphism.
Let $\pi: F^{-1} \twoheadrightarrow \Coker u $ be the natural epimorphism in the Gorenstein flat-cotorsion resolution
of $\Coker u$.
Then there exists an $R$-homomorphism $\eta : F^{-1}\rightarrow X$ such that $\pi =\rho _*(\eta) = \rho\circ\eta$.
Set $$\widetilde{b} = (\eta,    u\circ(M \otimes \eta),   u\circ (M \otimes u)\circ(M ^{2}\otimes \eta), \dots , u\circ (M \otimes u)\circ \cdots \circ(M ^N\otimes \eta)),$$
which is an $R$-homomorphism from $\bigoplus_{i=0}^N (M^{i} \otimes _R F^{-1})$ to $X$.
It is routine to check that the following diagram
\[
\begin{tikzcd}
0\arrow[r]&\bigoplus_{i=1}^N (M^{i} \otimes_R F^{-1})\arrow[r,  "c_{F^{-1}}"]\arrow[d,  "M \otimes \widetilde{b} "]&\bigoplus_{i=0}^N (M^{i} \otimes_R F^{-1})\arrow[r,  "\pi _{F^{-1}}"]\arrow[d,  "\widetilde{b}"]&F^{-1}\arrow[r]\arrow[d,  "\pi"]&0\\
0\arrow[r]&M \otimes_R X \arrow[r,  "u"]&X\arrow[r,  "\rho"]&\operatorname{Coker} u\arrow[r]&0
\end{tikzcd}
\]  commutes, and thus $\widetilde{b}$
forms a $T_R(M)$-homomorphism from $\Ind F^{-1}$ to $(X, u)$. Now we claim that
$\widetilde{b}$ is an epimorphism. 
Applying  the functor \( M^j \otimes_R - \) to the above diagram for each $0 \leq j \leq N $, one gets the exact commutative diagram
\[{\fontsize{10}{10}
\begin{tikzcd}
	 & \bigoplus_{i=j+1}^N (M^i \otimes_R F^{-1}) \arrow[r] \arrow[d,   "M^{j+1} \otimes\widetilde{b}"] & \bigoplus_{i=j}^N (M^i\otimes_R F^{-1}) ~~\arrow[r,  "M^j\otimes {\pi _{F^{-1}}}"] \arrow[d,   " M^j \otimes\widetilde{b}"] &M^j \otimes_R F^{-1} \arrow[r] \arrow[d,   "M^j \otimes \pi"] & 0 \\
	0 \arrow[r] & M^{j+1} \otimes_R X \arrow[r] & M^j \otimes_R X \arrow[r,  "{M^j\otimes\rho}"] & M^j \otimes_R \operatorname{Coker} u \arrow[r] & 0,
\end{tikzcd}}
\] where the exactness of the second row follows from the proof of Lemma~\ref{lem-ext-zero}.
Since \( M^j \otimes_R - \) is right exact, we have that $M^j\otimes {\pi}$ is an epimorphism for each $0 \leq j \leq N $.
Then it follows from the fact $M^{N+1} = 0$ that $M^N \otimes\widetilde{b}$ is an epimorphism,
and so are $M^{N-1} \otimes\widetilde{b}, \cdots, M \otimes\widetilde{b}$ and $\widetilde{b}$ by Snake Lemma.
So, there is a short exact sequence of $T_R(M)$-modules
$$  \xi :\ 0\rightarrow (Z,  \delta) \xrightarrow{\alpha} \Ind F^{-1}\xrightarrow{\widetilde{b}} (X,u)\rightarrow 0 $$
with $(Z,  \delta) = \ker \tilde{b}$. Set
$K := \ker \pi \cong \ker d^{-1}$. Then $K\in \GFC (R)$ since it is the (-1)-th cocycle of $F^\bullet$.  
Note that the sequence $\xi$
can be written as the following commutative diagram with exact rows and columns:
\[\begin{tikzcd}
	&&0\arrow[d]&0\arrow[d]&\\
	&M\otimes_RZ\arrow[r,  "M\otimes\alpha"]\arrow[d,  "\delta"] &\bigoplus_{i=1}^N (M^i\otimes_RF^{-1})\arrow[d,  "c_{F^{-1}}"]\arrow[r,  "M\otimes {\tilde{b}}"]  &M\otimes_R X  \arrow[r] \arrow[d,  "u"]& 0\\
	0\arrow[r]&Z\arrow[r,  "\alpha"]\arrow[d]&\bigoplus_{i=0}^N (M^i\otimes_RF^{-1})\arrow[r,  "\widetilde{b}"]\arrow[d,  "\pi_{F^{-1}}"]&X\arrow[r]\arrow[d,  "\rho"]&0\\
	0\arrow[r]&K\arrow[r]&F^{-1}\arrow[d]\arrow[r,  "\pi"]&\operatorname{Coker} u\arrow[d]\arrow[r]&0.\\
	&&0&0
\end{tikzcd}
\] On the other hand, it follows from
\cite[Lemma 2.9 ]{DLST} that $\operatorname{Tor}_1^R(M,   X) = 0$, and then $M\otimes \alpha$ is a monomorphism.
Therefore, we get
$M\otimes_R Z \cong \ker( M \otimes \widetilde{b} )$, and by Snake Lemma, we have an exact sequence $0 \to M \otimes{Z} \xrightarrow{\delta} Z \to K \to 0$.
Then the $T_R(M)$-module $(Z,  \delta)$ has the property that $\delta$
is a monomorphism  and $ \Coker(\delta)\cong K \in \GFC(R) $.
Repeating the above process, we obtain an exact sequence $$\cdots  \rightarrow \Ind F^{-2} \rightarrow \Ind F^{-1} \xrightarrow{\widetilde{b}}  (X,   u) \rightarrow 0$$
in $\Mod(T_R(M))$.  

Above all,  we obtain a long exact sequence of flat-cotorsion $T_R(M)$-modules
\[
\Delta^\bullet: \quad \cdots \longrightarrow \operatorname{Ind} F^{-2} \longrightarrow \operatorname{Ind} F^{-1} \xrightarrow{{\widetilde{a}}\circ\widetilde{b}} \operatorname{Ind} F^0 \longrightarrow \operatorname{Ind} F^1 \longrightarrow \cdots,
\]
whose zeroth cocycle is isomorphic to $(X,   u)$. Therefore, to show $(X,   u)\in \GFC(T_R(M))$, it remains to prove that
$\Hom_{T_R(M)}(\Delta^\bullet,   W)$ is exact for any $ W \in \FC(T_R(M))$,
and by Lemma~\ref{temsor-FC1},  it suffices to show $\Hom_{T_R(M)}(\Delta^\bullet,   \Ind(\tilde{W}))$ is exact for any $\tilde{W} \in \FC(R)$.
We note that we cannot use the isomorphism $\Hom_{T_R(M)}(\Delta^\bullet,   \Ind(\tilde{W}))\cong \Hom_{R}(F^\bullet,  \bigoplus_{i=0}^N (M^i\otimes_R\tilde{W}))$,
because the differentials of $\Delta^\bullet$ are not diagonal matrixes.
So, for any $0\leq i\leq N$, we consider the following 
short exact sequences of $T_R(M)$-modules
$$ 0 \to \Ind( M^{i+1} \otimes_{R} \tilde{W} ) \to \Ind( M^{i} \otimes_{R}\tilde{W}) \to ( M^{i} \otimes_{R}\tilde{W},0) \to 0 $$
arising from \cite[Remark 1.5]{DLST}. 
Since $M\otimes _R-$ preserves cotorsion modules,
we get that $M^{i} \otimes_{R}\tilde{W}$ is cotorsion for every $i$, and so is $\Ind( M^{i} \otimes_{R} \tilde{W} )$
by Lemma~\ref{tensor-cor}. Therefore, we have $\Ext_{T_R(M)}^1( \Ind F^j,   \Ind( M^{i+1} \otimes_{R} \tilde{W} ) ) =0$,
and then we get a short exact sequence of complexes
\[
0 \to(\Delta^\bullet,  \operatorname{Ind}(M^{i+1}\otimes_R\tilde{W}))  \to(\Delta^\bullet,  \operatorname{Ind}(M^i\otimes_R\tilde{W}))\to (\Delta^\bullet,  ( M^{i} \otimes_{R}\tilde{W},0)) \to 0
\] where $(-,-):=\Hom_{T_R(M)}(-,   -)$. 
When $i=N$, we have $(\Delta^\bullet,  \operatorname{Ind}(M^N\otimes_R\tilde{W}))\cong(\Delta^\bullet,  ( M^{N} \otimes_{R}\tilde{W},0))$.
Thus, it suffices to show that $(\Delta^\bullet,  (M^i\otimes_R \tilde{W},0))$ is exact for all $i\geq 0$.
Once this step is complete, we conclude that the complex
$(\Delta^\bullet,  \operatorname{Ind}(M^N\otimes_R\tilde{W}))$ is exact,
and it follows from the aforementioned short exact sequence
that the complexes $(\Delta^\bullet,  \operatorname{Ind}(M^{N-1}\otimes_R\tilde{W})), \cdots, (\Delta^\bullet,  \operatorname{Ind}(M\otimes_R\tilde{W})),
(\Delta^\bullet,  \operatorname{Ind}(\tilde{W}))$  are exact as well.

To show the exactness of $\Hom _{T_R(M)}(\Delta^\bullet,  (M^i\otimes_R \tilde{W},0))$, we first
claim that the complex $\Hom_{T_R(M)}(\Delta^\bullet,   (Y,0))$ is isomorphic to
$\Hom_{R}(F^\bullet,   Y)$ for any $Y\in \Mod (R)$. Consider the following diagram 
$$\begin{tikzcd}
( \Delta^\bullet,  (Y,0) ):\ 	\cdots\arrow[r]&(\Ind F^{0},  (Y,0))\arrow[r,  "(\widetilde{a}\circ\widetilde{b})^\ast"]\arrow[d,  "\Phi^{0}"]&(\Ind F^{-1},  (Y,0))\arrow[r]\arrow[d,  "\Phi^{-1}"]&\cdots\\
	( F^\bullet,  Y ):\  	\cdots\arrow[r]&\arrow[r](F^{0},  Y)\arrow[r,  "(d^{-1})^\ast"]&(F^{-1},  Y)\arrow[r]&\cdots ,\\
	\end{tikzcd} (\ast)$$
where $(-,-):=\Hom_{T_R(M)}(-,   -)$ and the vertical maps are the adjoint isomorphisms, that is,
$\Phi ^0(\alpha)=\alpha^0$ for any $T_R(M)$-homomorphism $\alpha=(\alpha^0,  \alpha^1,  \alpha^2,  \cdots,  \alpha^N)
:\Ind F^{0}\rightarrow (Y,0)$ with $\alpha^i:M^i \otimes_R F^{0}\rightarrow Y$. 
Now we claim that the diagram $(\ast)$ is commutative. Fix a $T_R(M)$-homomorphism $l=(l^0,l^1,\cdots, l^N)
\in \Hom_{T_R(M)}(\Ind F^{0}, (Y,0))$ with $l^i\in \Hom_R(M^i \otimes_R F^{0}, Y)$.
Then we have the following commutative diagram
$$\begin{tikzcd}
\bigoplus_{i=1}^N (M^i \otimes_R F^{-1})\arrow[r,  "c_{F^{0}}"]\arrow[d,  "M\otimes l"]&\bigoplus_{i=0}^N (M^i \otimes_R F^{0})\arrow[d,  "l"]\\
M\otimes_R Y\arrow[r,  "0"]&Y,\\
\end{tikzcd}$$
and thus
$$0=l\circ c_{F^{0}}=
	(l^0,l^1,l^2,\cdots,l^N)
\begin{pmatrix}
		0 & 0 & 0 & \cdots & 0\\
	1 & 0 & 0 & \cdots & 0\\
	0 & 1 & 0 & \cdots & 0\\
	\vdots & \vdots & \vdots & \ddots & \vdots\\
	0 & 0 & 0 & \cdots & 1
	\end{pmatrix}
=(l^1,l^2,\cdots,l^N).$$
Therefore, we get $l^i=0$ for any $1\leq i\leq N$, and then
$l=	(l^0,0,0,\cdots,0)$. On the other hand, we have
$[\Phi^{-1} \circ (\widetilde{a} \circ \widetilde{b})^{*}](l) = \Phi^{-1}[(\widetilde{a} \circ \widetilde{b})^{*}(l)]
= \Phi^{-1}(l \circ \widetilde{a} \circ \widetilde{b})$, where $l \circ \widetilde{a} \circ \widetilde{b}$ is a map from $\Ind F^{-1}$ to $(Y,0) $,
and $\Phi^{-1}(l \circ \widetilde{a} \circ \widetilde{b})$ is the restricted map from $F^{-1}$ to $Y$.
Then the construction of $\widetilde{b}$ yields that $$\Phi^{-1}(l \circ \widetilde{a} \circ \widetilde{b})
=l \circ \widetilde{a} \circ \eta=(l^0,0,0,\cdots,0)\circ \begin{pmatrix}
		a^0 \\
		a^1 \\
		a^2 \\
		\vdots \\
		a^N
	\end{pmatrix}\circ \eta=l^0\circ a^0\circ\eta,
$$ and then $\Phi^{-1}(l \circ \widetilde{a} \circ \widetilde{b})=l^0\circ(i\circ \rho)\circ\eta
=l^0\circ i\circ (\rho \circ\eta)=l^0\circ i\circ \pi = l^0\circ d^{-1}$.
Since $\left[(d^{-1})^{*} \circ \Phi^{{0}}\right](l) = (d^{-1})^{*}(l^{0}) = l^{0} \circ d^{-1}$, we get that
the diagram $(\ast)$ is commutative, and then the complex $\Hom_{T_R(M)}(\Delta^\bullet,   (Y,0))$ is isomorphic to
$\Hom_{R}(F^\bullet,   Y)$ for any $Y\in \Mod (R)$. Therefore,
the complex $\Hom _{T_R(M)}(\Delta^\bullet,  (M^i\otimes_R \tilde{W},0))$
is isomorphic to
$\Hom_{R}(F^\bullet,  M^i\otimes_R \tilde{W})$, which is exact as $M$ is admissible. 
Indeed, since $M^i\otimes_R \tilde{W}$ is cotorsion and the class of flat
modules and cotorsion modules form a complete hereditary
cotorsion pair, we have $\Ext _R^{\geq 1}(F^j, M^i\otimes_R \tilde{W})=0$ for any $j$.
Then for any $s\in \mathbb{Z}$, to compute $\operatorname{Ext}_{R}^1(\Im d^{s+1},   M^i\otimes_R\tilde{W})$, we
consider the exact sequence
$$\cdots \rightarrow F^s\xrightarrow{d^s} F^{s+1} \twoheadrightarrow \Im d^{s+1}\rightarrow 0$$
and apply $\Hom _R(-,M^i\otimes_R\tilde{W})$. Then we have
\[
H^s(\operatorname{Hom}_R( F^\bullet,  M^i\otimes_R \tilde{W}))\cong\operatorname{Ext}_{R}^1(\Im d^{s+1},   M^i\otimes_R\tilde{W}),
\] which is equal to zero since $M$ is admissible, $\Im d^{s+1}\in \GFC(R)$ and $\tilde{W}\in FC (R)$.
\end{proof}

As an immediate consequence of Proposition~\ref{com-adm}, Proposition~\ref{prop:compatible}
and Proposition~\ref{prop:admissible},
we have the next result.

\begin{theorem}\label{cor-ten-GFC}
Suppose that $M$ satisfies the condition ${\rm(P)}$ and $M \otimes_R-$
preserves cotorsion modules. If $\fd_R M < \infty$ and $\fd _{R^{op}}M < \infty$, then
 $(X,  u)\in \GFC(T_R(M))$  if and only if $u$ is monomorphic and $\Coker u \in \GFC(R)$. 
\end{theorem}

The following example is due to Cui, Rong and Zhang \cite[Example 4.7]{CRZ}, which
shows that there exists a bimodule $M$ over some algebra satisfying all the conditions in the above theorem.

\begin{example}\label{ex}
{\rm Let $k$ be a field and $kQ$ the path algebra associated to the quiver
\[
Q: \begin{tikzcd}[column sep=1.5em, ampersand replacement=\&]
1 \ar[r] \& 2 \ar[r] \& 3 \ar[r] \& \cdots \ar[r] \& n. \ar[llll, bend right]
\end{tikzcd}
\]
Suppose that $J$ is the ideal of $kQ$ generated by all the arrows. Then $R = kQ/J^h$ is a self-injective algebra for $2 \leqslant h \leqslant n$. Denote by $e_i$ the idempotent element corresponding to the vertex $i$. Then one has $e_j R e_i = 0$ whenever $1 \leqslant i < j \leqslant n$ and $j - i \geqslant h$. Let $M = R e_i \otimes_k e_j R$. Then $M$ is an $R$-bimodule,
which is finitely-generated projective as a right $R$-module. Therefore,
$M$ satisfies the condition ${\rm(P)}$ and $M \otimes_R-$
preserves cotorsion modules; see Lemma~\ref{lem-pre-cor}.
Moreover, $M$ is nilpotent since $M \otimes_R M \cong R e_i \otimes_k (e_j R \otimes_R R e_i) \otimes_k e_j R = 0$.}
\end{example}

\section{Applications}
\indent\indent In this section, we present some applications to trivial ring extensions and Morita context rings.

\noindent\textbf {4.1 The trivial extension of rings.}
Let $R$ be a ring and $M$ an $R$-bimodule. The {\it trivial extension} ring of $R$ by $M$,
defined as $ R \ltimes M $, is the Cartesian product $R \times M$
with the natural addition and multiplication given by
\[
(r_1, m_1)(r_2, m_2) = (r_1r_2,\; r_1m_2 + m_1r_2),
\]
for all $ r_1, r_2 \in R $ and $m_1, m_2 \in M $. We refer to \cite{FGR, RI} for more details
on trivial extension rings. 

Suppose that the $R$-bimodule $M$ is $1$-nilpotent, that is, $M \otimes_R M = 0.$ Then the tensor ring $ T_R(M) $ is nothing but the
trivial extension ring $ R \ltimes M.$ One can immediately get the following results by Theorem~\ref{cor-ten-GFC}.
\begin{corollary}\label{cor-tri-exten}
Suppose that $M $ is a $1$-nilpotent $R$-bimodule. If 
$\Tor _{\geq 1}^R(M,M)$ $=0$,
$\fd_R M < \infty$ , $\fd _{R^{op}}M < \infty$ and $M \otimes_R-$
preserves cotorsion modules, then $(X,  u)\in \GFC(R \ltimes M)$ if and only if $u$ is monomorphic and $\Coker u \in \GFC(R)$. 
\end{corollary}

\begin{remark}\label{remk-tri-exten}
{\rm In \cite{WHJ}, Wang et al. established some sufficient and necessary
conditions for Gorenstein flat-cotorsion modules over
$ R \ltimes M$. When $M$ is $1$-nilpotent, 
the sequence $M\otimes _RM\otimes _RX\xrightarrow{M\otimes u} M\otimes _RX\xrightarrow{u} X$
is exact if and only if $u$ is a monomorphism. In this case, 
the conditions in Corollary~\ref{cor-tri-exten} improve upon those in \cite[Theorem 2.6]{WHJ}.
Indeed, we weaken the assumption that $M_R$ is flat by only requiring that $\fd (M_R)< \infty$ and $\Tor _{\geq 1}^R(M,M)=0$. 
}
\end{remark}

\noindent\textbf {4.2 Morita context rings.}
Let $A$, $B$ be rings, $N$ an $A$-$B$-bimodule and $M$ a $B$-$A$-bimodule. Consider the Morita context ring
$$ \Lambda = \begin{pmatrix} A & N \\ M & B \end{pmatrix} $$
with zero bimodule homomorphisms, i.e. multiplication is given by
$$ \begin{pmatrix} a & n \\ m & b \end{pmatrix} \begin{pmatrix} a' & n' \\ m' & b' \end{pmatrix} = \begin{pmatrix} aa' & an' + nb' \\ ma' + bm' & bb' \end{pmatrix}. $$
Recall that $\text{Mod}(\Lambda)$ is equivalent to a category with objects tuples $(X, Y, f, g)$ where $X \in \text{Mod}(A)$, $Y \in \text{Mod}(B)$, $f \in \text{Hom}_B(M \otimes_A X, Y)$ and $g \in \text{Hom}_A(N \otimes_B Y, X)$. For a detailed exposition of this we refer to \cite{G82, GP14}. 

It follows from \cite[Proposition 2.5]{GP14} that the ring $\Lambda$ is isomorphic to the trivial extension
ring $(A \times B) \ltimes (M \oplus N)$. We mention that $(M\oplus N) \otimes_{A \times B} (M \oplus N) \cong (M \otimes_A N) \oplus (N \otimes_B M)$.
Then the $A \times B$-bimodule $M \oplus N$ is $1$-nilpotent if and only if $M \otimes_A N = 0 = N \otimes_B M$. Thus, we obtain the next result by Corollary
~\ref{cor-tri-exten}.

\begin{corollary}\label{cor-morita-ring}
Let $\Lambda$ be a Morita context ring with $M \otimes_A N = 0 = N \otimes_B M$, and let $(X, Y, f, g)$ be a $\Lambda$-module.
Assume that both $M \otimes_A-$ and $N \otimes_B-$
preserve cotorsion modules.
If $\mathrm{Tor}_{\geqslant 1}^B(N, M ) = 0 = \mathrm{Tor}_{\geqslant 1}^A(M, N)$,
$\mathrm{fd}_B M < \infty$, $\mathrm{fd}_{A^{\mathrm{op}}} M < \infty$, $\mathrm{fd}_A N < \infty$ and $\mathrm{fd}_{B^{\mathrm{op}}} N < \infty$,
then $(X, Y, f, g) \in \mathsf{GFC}(\Lambda)$ if and only if both $f$ and $g$ are monomorphisms, and $\mathrm{Coker}(f) \in \mathsf{GFC}(B)$ and $\mathrm{Coker}(g) \in \mathsf{GFC}(A)$.
\end{corollary}

In the setting where $M \otimes_A N = 0 = N \otimes_B M$, Corollary~\ref{cor-morita-ring} refines \cite[Proposition 3.1]{WHJ},
as highlighted in Remark~\ref{remk-tri-exten}. Specifically, it weakens the flatness assumptions on
$M_A$ and $N_B$ by only requiring that $\fd (M_A)< \infty$,
$\fd (N_B)< \infty$, $\Tor _{\geq 1}^A(M,N)=0$ and $\Tor _{\geq 1}^B(N,M)=0$. 

\begin{example}
{\rm Suppose that the ring $R$ and the $R$-bimodule $M$ as in Example~\ref{ex}. Then the Morita context ring
\[
\Lambda = \begin{pmatrix}
R & M \\
M & R
\end{pmatrix}
\]
satisfies all conditions in Corollary~\ref{cor-morita-ring}.}
\end{example}

\noindent {\footnotesize {\bf ACKNOWLEDGMENT.}
This work is supported by
the National Natural Science Foundation of China (12561008),
the project of Young and Middle-aged Academic and Technological leader of Yunnan
(202305AC160005) and the Basic Research Program of Yunnan Province (202301AT070070).}

\end{document}